\documentclass[notitlepage]{amsart}

\usepackage[dvipsnames]{xcolor}
\usepackage{amsmath,amssymb,amsthm,amsfonts,enumerate,tikz,bm}
\usepackage{mathrsfs}
\usetikzlibrary{matrix,arrows,positioning,calc}
\usepackage[all]{xy}
\usepackage[bookmarksnumbered,colorlinks]{hyperref}
\usepackage{url}
\numberwithin{equation}{section}
\usepackage{graphicx}
\usepackage[T1]{fontenc}
\usepackage{textcomp}
\usepackage{palatino,helvet}

\usepackage{footmisc}
\usepackage{float}
\usepackage{caption}

\newtheorem{pro}{Proposition}[section]
\newtheorem{teo}[pro]{Theorem}
\newtheorem{defi}[pro]{Definition}
\newtheorem{lem}[pro]{Lemma}
\newtheorem{cor}[pro]{Corollary}
\newtheorem{rk}[pro]{Remark}

\newcommand{\Ext}{\mathrm{Ext}}
\newcommand{\Tor}{\mathrm{Tor}}

\newcommand{\Hom}{\mathrm{Hom}}

\newcommand{\A}{\mathcal{A}}

\newcommand{\F}{\mathcal{F}}

\newcommand{\Le}{\mathcal{L}}

\newcommand{\X}{\mathcal{X}}
\newcommand{\Y}{\mathcal{Y}}

\newcommand{\W}{\mathcal{W}}
\newcommand{\pd}{\mathrm{pd}}
\newcommand{\op}{\mathrm{op}}

\newcommand{\Gpd}{\mathrm{Gpd}}

\newcommand{\Proj}{\mathcal{P}}

\newcommand{\Inj}{\mathcal{I}}

\newcommand{\id}{\mathrm{id}}
\newcommand{\Gid}{\mathrm{Gid}}
\newcommand{\resdim}{\mathrm{resdim}}

\newcommand{\coresdim}{\mathrm{coresdim}}

\newcommand{\Modu}{\mathrm{Mod}}
\newcommand{\Ker}{\mathrm{Ker}}
\newcommand{\Ima}{\mathrm{Im}}
\newcommand{\GP}{\mathcal{GP}}

\newcommand{\GF}{\mathcal{GF}}
\newcommand{\Flat}{\mathcal{F}}

\newcommand{\fd}{\mathrm{fd}}
\newcommand{\Gfd}{\mathrm{Gfd}}

\newcommand{\GI}{\mathcal{GI}}

\newcommand{\Coker}{\mathrm{CoKer}}

\newcommand{\Ch}{\mathrm{Ch}}

\usepackage{latexsym,amssymb,amscd} 
\usepackage{amsmath}
\usepackage[all]{xypic}
\usepackage{amsthm}

\newcommand{\cogorro}{\vee}
\newcommand{\ortogonal}{\bot}
\newcommand{\gorro}{\wedge}

\begin{document}
\title[Strongly $FP$-injective dimensions and Gorenstein projective precovers]{Strongly $FP$-injective dimensions and Gorenstein projective precovers}

\author{V\'ictor Becerril}
\address[V. Becerril]{Centro de Ciencias Matem\'aticas. Universidad Nacional Aut\'onoma de M\'exico. 
 CP58089. Morelia, Michoac\'an, M\'EXICO}
\email{victorbecerril@matmor.unam.mx}
\begin{abstract} 
The existence of the Gorenstein projective precovers over $R$ an arbitrary ring, as well as the completeness of the Gorenstein projective cotorsion pair $(\GP , \GP ^{\ortogonal})$, are open questions. In this paper, we provide some answers to these questions and use the tool developed to confirm the Gorenstein Symmetry Conjecture, under two different situations.   We also analyze situations where the finiteness of $\mathrm{silp} (R)$ implies $\mathrm{spli} (R) $ finite, and its counterpart.
\end{abstract}  
\maketitle

\section{Introduction} 
  For $R$ an associative ring  with identity, the Gorenstein projective, injective and flat $R$-modules where introduced in \cite{EnJen93}, since then the Gorenstein homological algebra has been developing intensively as a relative version of homological algebra that replaces the classical projective, injective and flat modules and resolutions with the Gorenstein versions. While in the classical homological algebra is know the existence of projective resolutions (resp. injective and flat resolutions)  for any $R$-module without restrictions over $R$,  the situation is different for the Gorenstein version.  The question: \textit{What is the most general class of rings over which all modules have Gorenstein projective (injective) resolutions?} still open. Is know that all module of finite Gorenstein projective dimension has a special Gorenstein projective precover. Also, if all Gorenstein projective is Gorenstein flat and all Gorenstein flat has finite Gorenstein projective dimension, then the class of Gorenstein projective modules is special precovering \cite{EA17}.   Also, if all finitely presented has finite Gorenstein projective dimension $\leq n$,  then $\GP$ is special precovering \cite{Estrada23}. Furthermore the pair $(\GP(R) , \GP (R)^{\ortogonal})$ is an hereditary and complete cotorsion pair.   Although it was recently proven that   the pair $(\GP(R) , \GP (R)^{\ortogonal})$ is always an hereditary cotorsion pair \cite{Cortes}, the situation of being complete remains unclear. Another question proposed by M. Cort\'es-Izurdiaga, J. Šaroch \cite{Cortes} is: \textit{What classes of modules appear in the cotorsion pair generated  by these modules?}
  In this paper we give a partial response to each one of this questions. 

An Artin algebra $A$ is called  \textit{Gorenstein} if is self-injective on both sides and these dimensions coincide. To date, no Artin algebra with distinct self-injective dimensions is known, so it makes sense in representation theory the famous \textit{Gorenstein Symmetry Conjecture} which states  that for an Artin algebra $A$ the left and the right injective dimension of the regular module always coincide. The Gorenstein symmetry conjecture is open in general but true for left weakly Gorenstein algebras, as shown by C. M. Ringel, P. Zhang  \cite[Questions 9.3]{Ringel} and R. Marczinzik \cite[Proposition 3.9 (2)]{Mar}. 

Another open question concerns  invariants $\mathrm{silp} (R)$ and $\mathrm{spli}(R)$ (the sumpremum of the projective lengths of injectives, or $\pd (\Inj (R))$) which are known are equal, if they are both finite by A. Beligiannis and I. Reiten \cite[Chapter VII . Prop. 1.3 (vi)]{Bel-Rei}. Nevertheless, as Gedrich
and Gruenberg point out in \cite{Gedrich}, \textit{it is not clear whether the finiteness of one of these implies the finiteness of the other}, i.e. whether we always have an equality $\mathrm{silp} (R) = \mathrm{spli}(R)$.

In the special case where $R$ is an Artin algebra, the equality $\mathrm{silp} (R) = \mathrm{spli}( R) $ is equivalent to the Gorenstein Symmetry Conjecture \cite[Conjecture 13]{AuRe95}, \cite[§11]{Bel} and \cite[Chapter VII]{Bel-Rei}. Therefore, in this paper we focus on analysing when $\mathrm{silp} (R) < \infty$ implies $\mathrm{spli} (R) < \infty$, and also when  $\mathrm{spli} (R) < \infty$ implies $\mathrm{silp} (R) < \infty$.  
In fact, we give two situations where this occurs, specifically; If every left injective has finite flat dimension, then $\mathrm{silp} (R) < \infty$ implies $\mathrm{spli} (R) < \infty.$ Furthermore, if the class of left projectives has finite Strongly FP-injective dimension, then $\mathrm{spli} (R) < \infty$ implies $\mathrm{silp} (R) < \infty$. From where we confirm the Gorenstein Symmetry Conjecture in both situations. 

\section{Preliminaries}
In what follows, we shall work with left modules over an associative ring $R$ with identity, denoted $\Modu (R)$.
Projective, injective and flat $R$-modules will be important to present some definitions, remarks and examples. The classes of projective left and right $R$-modules will be denoted by $\mathcal{P}(R)$ and $\mathcal{P}(R^{\rm op})$, respectively. Similarly, we shall use the notations $\mathcal{I}(R)$, $\mathcal{I}(R^{\rm op})$, $\mathcal{F}(R)$ and $\mathcal{F}(R^{\rm op})$ for the classes of injective and flat modules in $\Modu(R)$ and $\Modu(R^{\rm op})$, respectively. 
Concerning functors defined on modules, $\Ext^i_R(-,-)$ denotes the right $i$-th derived functor of $\Hom_R(-,-)$. If $M \in \Modu(R^{\rm op})$ and $N \in \Modu(R)$, $M \otimes_R N$ denotes the tensor product of $M$ and $N$. 
\subsection*{Orthogonality}

Let $\mathcal{X} \subseteq \Modu(R)$, $i \geq 1$ be a positive integer and $N \in \Modu(R)$. The expression $\Ext^i_R(\mathcal{X},N) = 0$ means that $\Ext^i_R(X,N) = 0$ for every $X \in \mathcal{X}$. Moreover, $\Ext^i_R(\mathcal{X,Y}) = 0$ if $\Ext^i_R(\mathcal{X},Y) = 0$ for every $Y \in \mathcal{Y}$. The expression $\Ext^i_R(N,\mathcal{Y}) = 0$ has a similar meaning. Moreover, by $\Ext^{\geq 1}_R(M,N) = 0$ we mean that $\Ext^i_R(M,N) = 0$ for every $i \geq 1$. One also has similar meanings for $\Ext^{\geq 1}_R(\mathcal{X},N) = 0$, $\Ext^{\geq 1}_R(N,\mathcal{Y}) = 0$ and $\Ext^{\geq 1}_R(\mathcal{X,Y}) = 0$. We can also replace $\Ext$ by $\Tor$ in order to obtain the notations for $\Tor$-orthogonality. The right orthogonal complements of $\mathcal{X}$ will be denoted by
\begin{align*}
\mathcal{X}^{\perp_i}  & := \{ M \in \Modu(R) {\rm \ : \ } \Ext^i_{R}(\mathcal{X},M) = 0 \}  & & \mbox{ and } &
\mathcal{X}^{\perp}  & := \bigcap_{i \geq 1} \mathcal{X}^{\perp_i}.
\end{align*}
The left orthogonal complements, on the other hand, are defined similarly.

\subsection*{Relative homological dimensions} 

There are homological dimensions defined in terms of extension functors. Let $M \in  \Modu (R)$ and $\mathcal{X}, \mathcal{Y} \subseteq  \Modu(R)$. The \emph{injective dimensions of $M$ and $\mathcal{Y}$ relative to $\mathcal{X}$} are defined by
\begin{align*}
\id_{\mathcal{X}}(M) &  := \inf \{ m \in \mathbb{Z}_{\geq 0} \text{ : } \Ext _R ^{\geq m+1}(\mathcal{X},M) = 0  \} & & \mbox{ and } &
\id_{\mathcal{X}}(\mathcal{Y}) & := \sup \{ \id_{\mathcal{X}}(Y) \text{ : } Y \in \mathcal{Y} \}.
\end{align*}
In the case where $\mathcal{X} =  \Modu(R)$, we write 
\begin{align*}
\id_{ \Modu(R)}(M) & = \id(M) & & \text{and} & \id_{\mathsf{Mod}(R)}(\mathcal{Y}) & = \id(\mathcal{Y})
\end{align*} 
for the (absolute) injective dimensions of $M$ and $\mathcal{Y}$.  Dually we can define the relative  dimensions $\pd _{\X} (M)$,  $\pd _{\X} (\Y)$ and $\pd (M)$, $\pd (\Y)$.

By an \emph{$\mathcal{X}$-resolution of $M$} we mean an exact complex 
\[
\cdots \to X_m \to X_{m-1} \to \cdots \to X_1 \to X_0 \to M \to 0
\]
with $X_k \in \mathcal{X}$ for every $k \in \mathbb{Z}_{\geq 0}$.  If $X_k = 0$ for $k > m$, we say that the previous resolution has \emph{length} $m$. The \emph{resolution dimension relative to $\mathcal{X}$} (or the \emph{$\mathcal{X}$-resolution dimension}) of $M$ is defined as the value
\[
\resdim_{\mathcal{X}}(M) := \min \{ m \in \mathbb{Z}_{\geq 0} \ \mbox{ : } \ \text{there exists an $\mathcal{X}$-resolution of $M$ of length $m$} \}.
\]
Moreover, if $\mathcal{Y} \subseteq  \Modu (R)$ then
\[
\resdim_{\mathcal{X}}(\mathcal{Y}) := \sup \{ \resdim_{\mathcal{X}}(Y) \ \mbox{ : } \ \text{$Y \in \mathcal{Y}$} \}
\]
defines the \emph{resolution dimension of $\mathcal{Y}$ relative to $\mathcal{X}$}. The classes of objects with bounded (by some $n \geq 0$) and finite $\mathcal{X}$-resolution dimensions will be denoted by
\begin{align*}
\mathcal{X}^\wedge_n & := \{ M \in  \Modu(R) \text{ : } \resdim_{\mathcal{X}}(M) \leq n \} & \text{and} & & \mathcal{X}^\wedge & := \bigcup_{n \geq 0} \mathcal{X}^\wedge_n.
\end{align*}
Dually, we can define \emph{$\mathcal{X}$-coresolutions} and the \emph{coresolution dimension of $M$ and $\mathcal{Y}$ relative to $\mathcal{X}$} (denoted $\coresdim_{\mathcal{X}}(M)$ and $\coresdim_{\mathcal{X}}(\mathcal{Y})$). We also have the dual notations $\mathcal{X}^\vee_n$ and $\mathcal{X}^\vee$ for the classes of $R$-modules with bounded and finite $\mathcal{X}$-coresolution dimension.

\subsection*{Approximations}

Given a class $\mathcal{X}$ of left $R$-modules  and $M \in \Modu (R)$, recall that a morphism $\varphi \colon X \to M$ with $X \in \mathcal{X}$ is an \emph{$\mathcal{X}$-precover of $M$} if for every morphism $\varphi' \colon X' \to M$ with $X' \in \mathcal{X}$, there exists a morphism $h \colon X' \to X$ such that $\varphi' = \varphi \circ h$. An $\X$-precover $\varphi $ is special if $\Coker (\varphi) = 0$ and $\Ker (\varphi) \in \X^{\ortogonal _1}$.

 A class $\mathcal{X} \subseteq \Modu(R)$ is (\emph{pre})\emph{covering} if every left $R$-module has an $\mathcal{X}$-(pre)cover. Dually, one has the notions of (\emph{pre})\emph{envelopes} and (\emph{pre})\emph{enveloping} and \emph{special (pre)enveloping} classes. 

\subsection*{Cotorsion pairs}

Two classes $\mathcal{X,Y} \subseteq \Modu (R)$ of left $R$-modules for a \emph{cotorsion pair} $(\mathcal{X,Y})$ if $\mathcal{X} = {}^{\perp_1}\mathcal{Y}$ and $\mathcal{Y} = \mathcal{X}^{\perp_1}$. 

A cotorsion pair $(\mathcal{X,Y})$ is:
\begin{itemize}
\item \emph{Complete} if $\mathcal{X}$ is special precovering, that is, for every $M \in \Modu(R)$ there is a short exact sequence $0 \to Y \to X \to M \to 0$ with $X \in \mathcal{X}$ and $Y \in \mathcal{Y}$; or equivalently, if $\mathcal{Y}$ is special preenveloping. 

\item \emph{Hereditary} if $\Ext^{\geq 1}_R(\mathcal{X,Y}) = 0$; or equivalently, if $\mathcal{X}$ is resolving (meaning that $\mathcal{X}$ is closed under extensions and kernels of epimorphisms, and contains the projective left $R$-modules) or $\mathcal{Y}$ is coresolving.  
\end{itemize}
Note that if $(\mathcal{X,Y})$ is a hereditary cotorsion pair, then $\mathcal{X} = {}^{\perp}\mathcal{Y}$ and $\mathcal{Y} = \mathcal{X}^\perp$.



\section{Precovers from Strongly FP-injectives}
We recall that the class $\GP $ of \textbf{Gorenstein projective $R$-modules} consist of cycles of acyclic complexes of left projective $R$-modules which remains exact after applying the functor $\Hom _R (-, P)$ for all $P \in \Proj (R)$. In this section, we will study situations where the class $\GP$  is special precovering and when $(\GP, \GP ^{\ortogonal})$ is a complete cotorsion pair. Of course the completeness of such pair implies that  $\GP$  is special precovering. Actually is  know by M. Cort\'es-Izurdiaga, J. Šaroch \cite{Cortes} that the pair $(\GP(R) , \GP (R)^{\ortogonal})$ is always an hereditary cotorsion pair, a condition to be complete is that all projective modules are $\lambda$-pure-injective for some infinite regular cardinal $\lambda$, but the conditions to be complete still open. In this section, we provide conditions that do not involve cardinal numbers. 

To accomplish this, and following \cite{Guan} we say that an $R$-module $E$ is \textbf{strongly FP-injective} if the group $\Ext^{k} _R (F,E)$ is trivial for all $k \geq 1$ and all finitely presented $R$-module $F$. We denote by $\mathcal{SFI}$ to the class of all  strongly $FP$-injectives. Since $(^{\ortogonal}\mathcal{SFI} , \mathcal{SFI})$ is an hereditary and complete cotorsion pair \cite[Theorem 3.4]{Guan}, it is routine to check that $\coresdim_{\mathcal{SFI}}(M) \leq m$ if and only if $\Ext ^{i+m} _R (F, M) =0$ for all finitely presented $F$ and $i \geq 1$. For convenience we denote  such dimension by $\mathrm{SFI\mbox{-}id} (M)$.   Is also clear that $\mathrm{SFI\mbox{-}id} (M) \leq \id (M)$. Under these conditions, we have the following result. 
 
\begin{lem} \label{Hom-exacto}
Let $R$ be a ring and $\Le$ a class of left $R$-modules. An acyclic complex  of projective $R$-modules $X ^{\bullet}$ is  $\Hom _R (-,\Le)$ exact for every subclass $\Le \subseteq \mathcal{SFI} ^{\cogorro}$. In particular, if $\Proj (R) \subseteq \mathcal{SFI} ^{\cogorro}$ then every acyclic complex  of projective $R$-modules is $\Hom _R (-,\Proj (R))$-exact.
\end{lem}

\begin{proof}
We will make use of  \cite[Lemma 3.10 (b)]{BMS}, from where is enough that $\Ext _R ^{1} (Z_n (X ^{\bullet}), L) =0$ for all $n \in \mathbb{Z}$ and $L \in \Le$. Indeed, from \cite[Corollary 4.9 (ii)]{Emma24} we know  that  $\Ext _R ^{1} (Z_n (X ^{\bullet}), L) =0$ for all $n \in \mathbb{Z}$ and $L \in \mathcal{SFI} ^{\cogorro}$.
\end{proof}

It was recently proven by I. Emmanouil, O. Talelli  \cite[Proposition 2.1]{Emma25} that for each  $n \in \mathbb{N}$ the pair $(^{\ortogonal}( \mathcal{SFI} ^{\cogorro} _n), \mathcal{SFI} ^{\cogorro }_n)$ is an hereditary and complete cotorsion pair (in fact generated by a set). We will use such fact in the following result. Recall from J. Gillespie \cite{Gill16} that a \textbf{projective cotorsion pair} in $\Modu (R)$ is a complete cotorsion pair $(\F , \W) $ such that $\W$ is \textit{thick} \cite[Definition 2.5]{Gill16} and $\F \cap \W = \Proj (R)$. 
\begin{teo}\label{Gill-Advances}
If for some $n \in \mathbb{N}$ the hereditary and complete cotorsion pair $(^{\ortogonal}( \mathcal{SFI} ^{\cogorro} _n), \mathcal{SFI} ^{\cogorro }_n)$ is projective. Then, the hereditary cotorsion pair $(\GP , \GP ^{\ortogonal})$ is complete, furthermore they coincide.
\end{teo}
\begin{proof}
It follows  from  \cite[Remarks 4.10 (i)]{Emma24}   that  $\GP \subseteq {^{\ortogonal} \mathcal{SFI} ^{\cogorro }}$. Thus, if $G \in \GP$ then $\Ext^{1}  _R(G, L) =0 $ for all $L \in \mathcal{SFI} ^{\cogorro} _n$, which implies $G \in {^{\ortogonal} (\mathcal{SFI} ^{\cogorro} _n )}$. Now, since $(^{\ortogonal}( \mathcal{SFI} ^{\cogorro} _n), \mathcal{SFI} ^{\cogorro }_n)$ is projective, we obtain from \cite[Theorem 5.4]{Gill17} that $^{\ortogonal}( \mathcal{SFI} ^{\cogorro} _n) \subseteq \GP$.
\end{proof}

Note that in the conditions of  Theorem \ref{Gill-Advances}  we obtain that the right side $\GP ^{\ortogonal} = \mathcal{SFI} ^{\cogorro} _n$, which provides an answer to the description requested by M. Cort\'es-Izurdiaga, J. Šaroch \cite{Cortes}. In particular, we also obtain that $\GP$ is special precovering.

  In the following result, we will use the \textbf{finitistic strongly FP-injective dimension of $R$ }defined by 
  $$\mathrm{fin.}\mathcal{SFI} (R) : =  \sup \{\mathrm{SFI.id} (M ): M \in  \mathcal{SFI} ^{\cogorro } \}.$$
  We have the following characterization of projectivity. 
\begin{teo} \label{Caracteriza-Proyectivo}
Given $n \in \mathbb{N}$. The cotorsion pair $(^{\ortogonal}( \mathcal{SFI} ^{\cogorro} _n), \mathcal{SFI} ^{\cogorro }_n)$ is projective, if and only if  $\Proj (R) \subseteq \mathcal{SFI} ^{\cogorro}$ and $\mathrm{fin.}\mathcal{SFI} (R) \leq n$. 
\end{teo}

\begin{proof}
$\Rightarrow)$  Assume that $(^{\ortogonal}( \mathcal{SFI} ^{\cogorro} _n), \mathcal{SFI} ^{\cogorro }_n)$ is a projective cotorsion pair. By definition we know  $\Proj (R) \subseteq  {\mathcal{SFI} ^{\cogorro} _n}$.  Consider $N \in  \mathcal{SFI} ^{\cogorro} $, we only need to prove that $N \in  \mathcal{SFI} ^{\cogorro} _n$. From \cite[Corollary 4.9 (ii)]{Emma24} we know that $\Ext ^1 _R (G, N) =0$ for every $G \in \GP $, but from Theorem  \ref{Gill-Advances} we know that $\GP = {^{\ortogonal} \mathcal{SFI} ^{\cogorro} _n}$, that is $N \in { \mathcal{SFI} ^{\cogorro} _n}$.

$\Leftarrow)$  From $\Proj (R) \subseteq \mathcal{SFI} ^{\cogorro}$ and $\mathrm{fin.}\mathcal{SFI} (R) \leq n$ we obtain that $\Proj (R) \subseteq \mathcal{SFI } ^{\cogorro} _n$, also we know that $\Proj (R) \subseteq \GP \subseteq {^{\ortogonal} (\mathcal{SFI } ^{\cogorro} _n}) $. This implies that $$\Proj (R) \subseteq  {^{\ortogonal} (\mathcal{SFI } ^{\cogorro} _n})  \cap \mathcal{SFI} ^{\cogorro} _n.$$
Now take $M \in {^{\ortogonal} (\mathcal{SFI } ^{\cogorro} _n})  \cap \mathcal{SFI} ^{\cogorro} _n$. We can  consider an exact sequence $ \lambda : 0 \to K \to P \to M \to 0$ with $P \in \Proj (R)$. Since $P, M \in \mathcal{SFI } ^{\cogorro} _n$ we get $K \in \mathcal{SFI } ^{\cogorro}$, and  the condition $\mathrm{fin.}\mathcal{SFI} (R) \leq n$ implies that $K \in \mathcal{SFI } ^{\cogorro} _n = \mathcal{SFI } ^{\cogorro} $ (a similar argument proves that  $\mathcal{SFI } ^{\cogorro} _n$ is thick). Since $M \in  {^{\ortogonal}( \mathcal{SFI } ^{\cogorro} _n)}$  we obtain that $\lambda$ splits, thus $M | P$ and then $M \in \Proj (R)$.
\end{proof}
As direct consequence we obtain the following.
\begin{cor}
If  $\Proj (R) \subseteq \mathcal{SFI} ^{\cogorro}$ and $\mathrm{fin.}\mathcal{SFI} (R) \leq n$, then $(\GP , \GP ^{\ortogonal})$ is an hereditary and complete cotorsion pair. 
\end{cor}

As can be seen noted, the condition   $\Proj (R) \subseteq \mathcal{SFI} ^{\cogorro}$ is satisfied in particular  for rings with dimension $\mathrm{silp} (R)$ (supremum  of the injective lengths of projectives, or $\id (\Proj (R))$) finite, since we have the inequality   $\mathrm{SFI.id} (M )\leq \id (M)$, for all $M \in \Modu (R)$. 
\begin{cor}
Let $R$ be a ring and consider the following statements;
\begin{itemize}
\item[(i)] $R$ is a left perfect ring with  $\Proj (R) \subseteq \mathcal{SFI} ^{\cogorro} $ and $\mathrm{fin.}\mathcal{SFI} (R) \leq n$,
\item[(ii)] $^{\ortogonal}( \mathcal{SFI} ^{\cogorro} _n)$ is closed under direct limits, 
\item[(iii)] the cotorsion pair $(^{\ortogonal}( \mathcal{SFI} ^{\cogorro} _n), \mathcal{SFI} ^{\cogorro }_n)$ is perfect.
\end{itemize}
Then $\mathrm{(i)} \Rightarrow \mathrm{(ii)} \Rightarrow \mathrm{(iii)}$. Furthermore, if  $\Flat (R) \subseteq \mathcal{SFI} ^{\cogorro} $ and $\mathrm{fin.}\mathcal{SFI} (R) \leq n$. 
Then, $\mathrm{(iii)} \Rightarrow \mathrm{(i)} $, and thus all statements are equivalent.
\end{cor}
\begin{proof}
$\mathrm{(i)} \Rightarrow \mathrm{(ii)}$ Since $\Proj (R) \subseteq \mathcal{SFI} ^{\cogorro} $ and $\mathrm{fin.}\mathcal{SFI} (R) \leq n$, from Theorems \ref{Caracteriza-Proyectivo} and \ref{Gill-Advances} we know that  $\GP = {^{\ortogonal}( \mathcal{SFI} ^{\cogorro} _n)}$.  Thus, consider $\{G_j\} _{j \in J}$ a family of Gorenstein projectives. For each $j \in J$ there exists an exact sequence $0 \to G_j \to P_j ^{0} \to P_j ^{1} \to \cdots $ with $P_j ^{i} \in \Proj (R)$, from where we obtain the exact sequence 
$$0 \to \displaystyle {\lim _{\to}} G_j \to {\lim _{\to}} P_j ^{0}  \to {\lim _{\to}} P_j ^{1} \to \cdots . $$
Since  $R$ is perfect, then  $ \displaystyle {\lim _{\to}} P_j ^{i} \in \Proj (R)$.  Finally, the condition $\Proj (R) \subseteq \mathcal{SFI} ^{\cogorro} $ and Lemma \ref{Hom-exacto} implies $\displaystyle {\lim _{\to}} G_j  \in \GP = {^{\ortogonal}( \mathcal{SFI} ^{\cogorro} _n)}$.

$\mathrm{(ii)} \Rightarrow \mathrm{(iii)}$ Follows directly the application of \cite[Theorem 7.2.6]{EnJen00} on $(^{\ortogonal}( \mathcal{SFI} ^{\cogorro} _n), \mathcal{SFI} ^{\cogorro }_n)$. 

$\mathrm{(iii)} \Rightarrow \mathrm{(i)}$. Assume that $\Flat (R) \subseteq \mathcal{SFI} ^{\cogorro} $ and $\mathrm{fin.}\mathcal{SFI} (R) \leq n$. As before, from Theorems \ref{Caracteriza-Proyectivo} and   \ref{Gill-Advances} we know that  $\GP = {^{\ortogonal}( \mathcal{SFI} ^{\cogorro} _n)}$. Then, for $F \in \Flat (R)$ there exist a Gorenstein projective cover $0 \to K \to G \to F \to 0$ with $K \in \GP ^{\ortogonal} = \mathcal{SFI} ^{\cogorro} _n$. Since $F \in  \mathcal{SFI} ^{\cogorro} _n$,  then $G \in  \mathcal{SFI} ^{\cogorro} _n$. This implies that $G \in {^{\ortogonal} (\mathcal{SFI } ^{\cogorro} _n})  \cap \mathcal{SFI} ^{\cogorro} _n $. It follows from Theorem \ref{Caracteriza-Proyectivo} that  $G \in \Proj (R)$, i.e. $G \to F$ is a projective cover. Therefore, $R$ is left perfect.  
\end{proof}
It is  know that  invariants $\mathrm{silp} (R)$ and $\mathrm{spli}(R)$ (the sumpremum of the projective lengths of injectives, or $\pd (\Inj (R))$) are equal, if they are both finite. Nevertheless, as Gedrich
and Gruenberg point out in \cite{Gedrich}, it is not clear whether the finiteness of one of these implies the finiteness of the other, i.e. whether we always have an equality $\mathrm{silp} (R) = \mathrm{spli}(R)$. In what follows we will see that the rings on which  $\Proj (R) \subseteq \mathcal{SFI}^{\cogorro} $ is satisfied,  the condition $\mathrm{silp} (R)$ finite, implies  $\mathrm{spli}(R)$ finite, but before we need an additional tool. The notation $\mathrm{gl.GPD}(R)$ denotes the global Gorenstein projective dimension of $R$, given by $\sup \{\Gpd (M): M\in \Modu (R)\},$  in where $\Gpd (M) : = \resdim _{\GP} (M)$.

\begin{teo} \label{GlobalP}
Let $R$ be a ring such that $\Proj (R) \subseteq \mathcal{SFI}^{\cogorro} $ and $\pd (\Inj (R)) < \infty$. Then, $\mathrm{gl.GPD}(R) < \infty$,  in consequence $\GP$ is a special precovering class.
\end{teo}

\begin{proof}
Take $M \in \Modu (R)$, and consider a projective and a injective resolution to $M$, 
$$ \cdots \to P_1 \to P_0 \to M \to 0, \;\;\mbox{ and } \;\;  0 \to M \to I_0 \to I_1 \to \cdots ,$$
for $P_{-1} = M = I_{-1}$ and  $i \geq 0 $ we can declare $C_i := \Ima (P_i \to P_{i-1} )$ and $K_i := \Ker (I_i \to I_{i+1})$ to obtain a decomposition of such resolutions in short exact sequences for each $i \geq 0$ as follows 
$$ 0 \to C_{i+1} \to P_i \to C_i \to 0  \;\;\mbox{ and } \;\; 0 \to K_i \to I_i \to K_{i+1} \to 0.$$
By doing the corpoduct of the first family and the product of the second family we get the exact sequences 
\[  0 \to \bigoplus _{i \in \mathbb{N}} C_{i+1} \to \bigoplus _{i \in \mathbb{N}} P_i \to  M \oplus (\bigoplus _{i \in \mathbb{N}} C_{i+1}) \to 0, \]

\[0 \to M \oplus (\prod _{i \in \mathbb{N}} K_{i+1} )\to \prod _{i \in \mathbb{N}} I_i \to   \prod _{i \in \mathbb{N}} K_{i+1} \to 0,
\]
and adding both of them we get the exact sequence 
$$0 \to M \oplus (C \oplus K) \to P \oplus I \to M \oplus (C \oplus K) \to 0,$$
with $C := \bigoplus _{i \in \mathbb{N}} C_{i+1}$, $K:= \prod _{i \in \mathbb{N}} K_{i+1} $, $P := \bigoplus _{i \in \mathbb{N}}  P_i$ and $I:= \prod _{i \in \mathbb{N}} I_i$. Take an exact sequence  $\cdots \to F_1 \to F_0 \to M\oplus (C \oplus K)$ by projectives $F_i$. Using Horseshoe's Lemma we obtain the following  commutative and exact diagram
$$\xymatrix{ L_{m_{}\;} \ar@{^{(}->}[d]  \ar@{^{(}->}[r]& F_{} \ar@{^{(}->}[d]  \ar@{>>}[r]& L_{m_{}\;} \ar@{^{(}->}[d]  \\
 F_{m-1} \ar[d]  \ar@{^{(}->}[r]& F_{m-1} \oplus F_{m-1} \ar[d]  \ar@{>>}[r]& F_{m-1} \ar[d]  \\
  \vdots \ar[d]  & \vdots \ar[d] & \vdots  \ar[d]  \\
  F_0 \ar@{>>}[d] \ar@{^{(}->}[r] & F_0 \oplus F_0 \ar@{>>}[r]\ar@{>>}[d] & F_0  \ar@{>>}[d] \\
 M\oplus (C \oplus K) \ar@{^{(}->} [r]& P \oplus I \ar@{>>}[r] & M\oplus (C \oplus K), 
}$$
note that $\pd (P \oplus I) = \pd (I) \leq m := \pd (\Inj (R)) $, since $P \in \Proj (R)$. Therefore $F \in \Proj (R)$.  Pasting the exact sequence $0 \to L_m \to F \to L_m \to 0$  repeatedly we obtain an acyclic complex by projectives $$ \cdots \to F \to F \to F \to \cdots $$
which from Lemma \ref{Hom-exacto} is $\Hom_R (- , L)$-exact for all $L \in \Proj (R) \subseteq \mathcal{SFI}^{\cogorro} $. Thus, $L_m \in \GP$ and from the exact sequence  
$$0 \to F_m \to \cdots \to F_1 \to F_0 \to M\oplus (C \oplus K) \to 0$$
 we get $\Gpd (M\oplus (C \oplus K)) \leq m$. In consequence $\Gpd (M) \leq \Gpd (M\oplus (C \oplus K)) \leq m$, as desired.
\end{proof}
\begin{rk}
In the conditions of Theorem \ref{GlobalP} and by \cite[Corollary 16 (1)] {EA17} we obtain that $(\GP ,\GP ^{\ortogonal})$ is an hereditary and complete cotorsion pair. 
\end{rk}

As stated earlier, whenever that $\Proj (R) \subseteq \mathcal{SFI}^{\cogorro} $, then the condition $\mathrm{silp} (R) < \infty$, implies  $\mathrm{spli}(R)< \infty$. In the
special case where $R$ is an Artin algebra, the equality $\mathrm{silp} (R) = \mathrm{spli}( R) $ is equivalent to the Gorenstein Symmetry Conjecture in representation theory;  \cite[Conjecture 13]{AuRe95}, \cite[§11]{Bel} and \cite[Chapter VII]{Bel-Rei}.

\begin{pro} \label{CorSpli}
Let $R$ be a ring with $\Proj (R) \subseteq \mathcal{SFI}^{\cogorro} $, and such that $\mathrm{spli} (R) := \pd (\Inj (R)) < \infty$, then:
\begin{itemize}
\item[(i)] The equality  $\mathrm{spli} (R) = \mathrm{silp} (R)$ is given.
\item[(ii)] If additionally $R$ is an Artin algebra, the Gorenstein Symmetry Conjecture is confirmed.
\end{itemize} 
\end{pro}

\begin{proof}
(i) Is followed by \cite[Corollary 2.7]{BennisGlobal}, since $\mathrm{gl.GPD}(R) $ is finite by Theorem \ref{GlobalP}. The equality was proven in \cite[Chapter VII . Prop. 1.3 (vi)]{Bel-Rei}. Note that $\Modu (R)$ is a Gorenstein category   in the sense of A. Beligiannis and I. Reiten \cite[Chapter VII . Definition 1.1]{Bel-Rei}.
 
 (ii) For the last statement when $R$ is an Artin algebra, we need refer to the discussion given by A. Beligiannis and I. Reiten in \cite[Chapter VII, after of Remark 2.7]{Bel-Rei}. 
\end{proof}

We now return to consider precovers. In the next result, we show that  Gorenstein projective precovers can be obtained with fewer conditions; however, the author does not know whether such precover is special. 
\begin{pro}
Let $R$ be a ring such that $\Proj (R) \subseteq \mathcal{SFI}^{\cogorro} $, then every $R$-module has a $\GP$-precover. 
\end{pro}

\begin{proof}
From Lemma \ref{Hom-exacto} we know that every acyclic complex of projectives is $\Hom (-,\Proj (R))$-exact. Now, from \cite[Proposition 7.3]{Yang14} we obtain that $\mathrm{dw}_{\mathrm{ex}} (\Proj (R))$ the class of all acyclic degreewise projective complexes of $R$-modules is precovering in the category $\Ch (R)$ of complexes of $R$-modules, since $(\Proj (R), \Modu (R))$ is a projective cotorsion pair cogenerated by a set (by the zero module). Now take $M \in \Modu (R)$ and consider the complex 
$$D (M) = \cdots \to 0 \to M \to M \to 0 \to \cdots ,$$
 which is concentrate in degrees  $-1$ and $0$ with $M$, and whose  only nonzero differential $d _{-1} $ is the identity of $M$. Then for $D (M)$ there exists a $\mathrm{dw}_{\mathrm{ex}} (\Proj (R))$-precover $g: P ^{\bullet} \to D (M)$. Of course 
 $$P ^{\bullet} = \cdots \to P^{-2} \to P^{-1} \to P^0 \to P^1 \to \cdots,$$
  is an exact complex of projective $R$-modules.  We set $G:= \Ker (P^0 \to P^{1})$, then for the comment in the first line $G \in \GP$, and $g$ induces a morphism $\tilde{g} : G \to M$. At this point the arguments given in the proof of \cite[Theorem A]{Yang14} can be adapted to prove that $\tilde{g}: G \to M$ is the desired $\GP$-precover of $M$.
\end{proof}
We can further restrict the last condition to effectively find special $\GP$-precovers, as we prove  in the following result.

\begin{teo} \label{construction-precover}
Let $R$ be a ring and $n \in \mathbb{N}$ such that $\Proj (R) \subseteq \mathcal{SFI}^{\cogorro} _{n-1}$. Then $\GP$ is special precovering. Furthermore the cotorsion pairs $(\GP , \GP ^{\ortogonal})$ and $( ^{\ortogonal}  (\mathcal{SFI}^{\cogorro} _{n} ),  \mathcal{SFI}^{\cogorro}  _{n} )$ coincide.  
\end{teo}

\begin{proof}
 Assume that $\Proj (R) \subseteq \mathcal{SFI}^{\cogorro} _{n} $. We will show that $\GP = {^{\ortogonal}  (\mathcal{SFI}^{\cogorro} _{n} )}$. Take $M \in {^{\ortogonal}  (\mathcal{SFI}^{\cogorro} _{n} )}$, then there exists an $\mathcal{SFI}^{\cogorro}  _{n} $-preenvelope $\alpha : M \to L  $, which can be chosen as a monomorphism, since  $( ^{\ortogonal}  (\mathcal{SFI}^{\cogorro} _{n} ),  \mathcal{SFI}^{\cogorro}  _{n} )$ is a complete cotorsion pair. Consider the short exact sequence $\theta : 0 \to K \to P ^0 \xrightarrow{\varphi}  L \to 0$ with $P^0 \in \Proj (R)$. Since $ L \in \mathcal{SFI}^{\cogorro}  _{n}$ and $P^0 \in \mathcal{SFI}^{\cogorro}  _{n-1}$ it follows that $K \in \mathcal{SFI}^{\cogorro}  _{n}$. This implies that $\Ext ^1 _R (M, K) =0$. Applying $\Hom _R (M,-)$ to $\theta$ we obtain the exact sequence
 $$\Hom _R (M,P^0) \xrightarrow{\varphi _{\ast}} \Hom _R (M,L) \to \Ext ^1 _R (M,K) =0. $$ 
 Thus, for $\alpha \in \Hom _R (M,L)$ there exists $\psi \in \Hom _R (M,P^0)$ such that $\varphi \psi = \alpha$, as we see in the following commutative diagram
 
 $$\xymatrix{ 
        &   & M \ar@{-->}[dl] _{ \psi} \ar@{^{(}->}[dr]|-{\alpha _{}}  \ar[r]^{f} & L'    \\
   K^{} \ar@{^{(}->}[r]  & P^0 \ar@{..>}[urr]|-{\phi}   \ar@{>>}[rr]_{\varphi} &  & L  \ar@{-->}[u]_{ \tau} }$$
   now take $f : M \to L'$ with $L' \in \mathcal{SFI}^{\cogorro}  _{n} $. Since $\alpha$ is $\mathcal{SFI}^{\cogorro}  _{n} $-preenvelope, then there exists a morphism $\tau : L \to L'$ such that $f = \tau \alpha$ and thus both triangles are commutative. We now define $\phi := \tau \varphi$ and note that $f = \tau \alpha = \tau \varphi \psi = \phi \psi $, that is, $f$ is factored through of $\psi$. Therefore $\psi : M \to P^0$ is a $\mathcal{SFI}^{\cogorro}  _{n} $-preenvelope of $M$.  Note that $\psi$ is a monomorphism, since $\alpha = \varphi \psi$ and $\alpha$ is mono. 
   
   Consider the exact sequence $\lambda : 0 \to M \xrightarrow{ \psi  }  P^{0} \to C \to 0$ where $C := \Coker (\psi )$. Such sequence is $\Hom _R (-, \mathcal{SFI}^{\cogorro}  _{n} )$-acyclic since $\psi $ is a $\mathcal{SFI}^{\cogorro}  _{n} $-preenvelope, then is $\Hom _R (-,\Proj (R))$ acyclic, since $\Proj (R) \subseteq \mathcal{SFI}^{\cogorro}  _{n-1}$. We will prove now that $C \in {^{\ortogonal} \mathcal{SFI}^{\cogorro}  _{n}}$. To do this take $H \in \mathcal{SFI}^{\cogorro}  _{n}$. Applying $\Hom _R (-,H)$ to $\lambda$ we obtain the exact sequence   
   
     $$\xymatrix{ 
     \Hom _R (P^{0}, H)    \ar@{>>}[r]^{(\psi )^{\ast}} & \Hom _R (M,H)  \ar[r] ^{\eta} &  \Ext _R ^1 (C,H)  \ar[r] & \Ext ^1 _R (P^0 , H) =0 ,}$$

     therefore $C \in {^{\ortogonal } \mathcal{SFI}^{\cogorro}  _{n} }$.
     Thus, we can repeat this construction to obtain an exact sequence $0 \to M \to P^0 \to P^1 \to \cdots$, with $P^{i} \in \Proj (R)$ and  such that is $\Hom _R (-, \Proj (R))$-acyclic. From were we obtain that $^{\ortogonal} {\mathcal{SFI}^{\cogorro}  _{n}} \subseteq \GP $. For the oposite containment we already know that $\GP \subseteq {^{\ortogonal} {\mathcal{SFI}^{\cogorro}}}$ from \cite[Remarks 4.10 (i)]{Emma24}. In consequence $\GP \subseteq {^{\ortogonal} {\mathcal{SFI}^{\cogorro}}} \subseteq {^{\ortogonal} {\mathcal{SFI}^{\cogorro} _n} }$.
 \end{proof}
The previous Theorem \ref{construction-precover} allows us to  obtain the equality $\GP ^{\ortogonal} = \mathcal{SFI} ^{\cogorro} _n$, which provides another  answer (with less conditions that Theorem \ref{Gill-Advances}) to the description requested by M. Cort\'es-Izurdiaga, J. Šaroch \cite{Cortes}.
\begin{cor}
Let $R$ be a ring and $n \in \mathbb{N}$ such that $\Proj (R) \subseteq \mathcal{SFI}^{\cogorro} _{n-1}$. Then the cotorsion pair $( ^{\ortogonal}  (\mathcal{SFI}^{\cogorro} _{n} ),  \mathcal{SFI}^{\cogorro}  _{n} )$ is projective. In consequence $\mathrm{fin.}\mathcal{SFI} (R) \leq n$.
\end{cor}
\begin{proof}
From Theorem \ref{construction-precover} and  \cite[Theorem 5.4]{Gill16} we get that $( ^{\ortogonal}  (\mathcal{SFI}^{\cogorro} _{n} ),  \mathcal{SFI}^{\cogorro}  _{n} )$ is projective, since always is complete. Now we can invoque Theorem \ref{Caracteriza-Proyectivo} to obtain $\mathrm{fin.}\mathcal{SFI} (R) \leq n$.
\end{proof}

Recall some clases that we will use now. The class of \textbf{Gorenstein flat $R$-modules} $\GF $ consist of cycles of acyclic complexes of left flat $R$-modules  which remains exact  after applying the functor  $I \otimes_R -$, for all $I \in \Inj (R^{\op})$.
Also, the class  $ \Proj \GF $ is the presented in  \cite[\S 4]{Stov} called the class of  \textbf{projectively coresolved Gorenstein flat $R$-modules} which consist of cycles of acyclic complexes of left projective $R$-modules that remains exact after applying the functor $(I \otimes _R -)$ for all $I \in \Inj (R^{\op})$. 

An interesting relation is given with the Gorenstein projective dimension and the projective dimension of a module $M$ with $\id (M) < \infty$, indeed the equality  $\Gpd (M) = \pd (M)$ is given, as was proven in \cite[Theorem 2.2]{Holm04-1}. Also if $\id (M) < \infty$ and  when $R$ is two sided coherent with  right finitistic projective dimension finite, then by  \cite[Theorem 2.6]{Holm04-1} the equality of the Gorenstein flat dimension $\Gfd (M) := \resdim_{\GF} (M)$ with the flat dimension $ \fd (M)$ is given. We extend such result as follows.
\begin{teo}
Let $R$ be a ring and $M \in \Modu (R)$ such that $\mathrm{SFI\mbox{-}id} (M) < \infty$. The following equalities are true. 
\begin{itemize}
\item[(i)] $\Gpd (M) = \pd (M)$.
\item[(ii)] $\Gfd (M) = \fd(M)$.
\end{itemize}
\end{teo}
\begin{proof}
(i) The inequality $\Gpd (M) \leq \pd (M)$ is always true. To see the inverse we can assume that $\Gpd (M) = n < \infty$. Thus by \cite[Lemma 2.17]{Holm06}  there exists an exact sequence $0 \to M \to P \to G \to 0$ with $\pd (M) = n $ and $G \in \GP $, since $ \mathrm{SFI\mbox{-}id} (M) < \infty$, then  $\Ext ^1 _R (G, M) =0$ by \cite[Remarks 4.10 (i)]{Emma24}, thus the exact sequence splits and $\pd (M ) \leq \pd (P) = \Gpd (M)$. 

(ii) As above $\Gfd (M) \leq \fd (M)$ is always true. Thus, we can assume that $\Gfd (M)= n < \infty$. An easy induction 
of  \cite[Theorem 4.11]{Stov} shows that for $M$, there exists an exact sequence $0 \to M \to F \to G \to 0$ with $\fd (F) \leq n$ and $G \in \Proj \GF \subseteq \GP $, where the last containment is given by \cite[Theorem 4.4]{Stov}. As before,   $\Ext ^1 _R (G, M) =0$ by \cite[Remarks 4.10 (i)]{Emma24}. This implies that $\fd (M) \leq \fd (F) \leq n = \Gfd (M)$.
\end{proof}

In what follows, we address the situation when $\mathrm{silp}(R) < \infty$ implies $\mathrm{spli} (R)< \infty$, which will complement the result in  Proposition \ref{CorSpli}. Furthermore, using the same argument as before, we will confirm the Gorenstein Symmetry Conjecture under different conditions.  To do this, we need some tools. Recall from \cite{Xu} that an $R$-module $M$ is \textbf{strongly cotorsion} if $\Ext ^{1} _R (F, M) =0$ for all $F \in \Flat (R)^{\gorro}$. We denote by $\mathcal{SC} $ the class of all strongly cotorsion $R$-modules.  Is clear that every injective is strongly cotorsion, furthermore every Gorenstein injective is also strongly cotorsion \cite[Proposition 1.1]{Chris}. 

\begin{lem} \label{AcyGI}
Let $R$ be a ring such that $\Inj (R) \subseteq \Flat 
(R) ^{\gorro}$, then every acyclic complex of injectives is $\Hom _R (\Inj (R), -)$-exact. 
\end{lem}
\begin{proof}
From \cite[Proposition 1.1]{Chris} we know that every cycle module in an acyclic complex of strongly cotorsion modules is strongly cotorsion. Since $\Inj (R) \subseteq \mathcal{SC} $, then every acyclic complex of injectives $I ^{\bullet}$ is an acyclic complex of strongly cotorsion modules with cycles strongly cotorsion. It follows from \cite[Lemma 3.1]{Wang-Li} and  the dual of \cite[Lemma 3.10 (b)]{BMS} that $I ^{\bullet}$  is $\Hom _R (\Flat(R) ^{\gorro},-)$-exact. Thus, the condition $\Inj (R) \subseteq \Flat (R) ^{\gorro}$ implies in particular that $I ^{\bullet}$ is also $\Hom _R (\Inj (R),-)$-exact.
\end{proof}
The notation $\mathrm{gl.GID}(R)$ denotes the global Gorenstein injective dimension of $R$, given by $\sup \{\Gid (M): M\in \Modu (R)\},$  in where $\Gid (M) : = \coresdim _{\GI} (M)$.
\begin{teo} \label{Global-Inj}
Let $R$ be a ring such that $\Inj (R) \subseteq \Flat 
(R) ^{\gorro}$ and $\mathrm{silp} (R) := \id (\Proj (R)) < \infty$, then $\mathrm{gl.GID} (R) < \infty$.
\end{teo}
\begin{proof}
The dual arguments given in Theorem \ref{GlobalP} works appliying Lemma \ref{AcyGI} instead  of Lemma \ref{Hom-exacto} and appliying $\mathrm{silp} (R) < \infty$ instead $\mathrm{spli} (R) < \infty$.
\end{proof}
 In consequence we have the following result.
 
 \begin{pro} \label{CorSilp}
Let $R$ be a ring with $\Inj (R) \subseteq \Flat 
(R) ^{\gorro}$ and $\mathrm{silp} (R) := \id (\Proj (R)) < \infty$, then:
\begin{itemize}
\item[(i)] The equality  $ \mathrm{silp} (R) = \mathrm{spli} (R) $ is given.
\item[(ii)] If additionally $R$ is an Artin algebra, the Gorenstein Symmetry Conjecture is confirmed.
\end{itemize} 
\end{pro}
 \begin{proof}
The arguments presented in the proof  of   Proposition \ref{CorSpli} can de adapted. 
 \end{proof}
The following result is easily obtained.
 \begin{cor}
 If $R$ is a ring such that  $\Proj (R) \subseteq \mathcal{SFI}^{\cogorro} $ and $\Inj (R) \subseteq \Flat 
(R) ^{\gorro}$, then 
 
 \[\mathrm{silp} (R) < \infty, \mbox{if and only if}, \mathrm{spli} (R) < \infty.\]
 \end{cor}


\section{Precovers from Semidualizing bimodules}

Below we provide an additional way to guarantee special $\GP$-precovers, for this, we use a generalized class of Gorenstein projective modules. We will also  obtain complete cotorsion pairs.  

  We  recall some facts and a definitions adapted to our setting. Let  $R$ and $S$  be  associative rings with identities. In \cite{Araya}  Araya, Takahashi and Yoshino introduced the notion semidualizing bimodules  $_S C _R$. Associated to a semidualizing bimodule $_S C _R$ we have the Auslander and Bass classes, introduced by Foxby in \cite{Foxby}. We denote the \textbf{Auslander class}  by $\A _C (S^{\op})$. It was recently proved by Huang in \cite[Corollary 3.4]{Huang} that $(\A_C (R), \A_C (R^{\op}) ^{\ortogonal})$ is a hereditary and perfect cotorsion pair\footnote{Note that Huang \cite{Huang} works with $_R C _S$, instead $_S C_R$.}. In particular such pair is hereditary and complete.
We recall the notion of $C$-injective $R$-modules  \cite[Definition 5.1]{HolmWhite} , which is the class $\mathcal{I}_C (R)$ consisting  by $R$-modules of the form $\Hom _S (C, I)$ where $I \in \Inj (S)$.
 
 We denote by $\GP _{\A_C}$ to the class of $(\Proj (R), \A _C (R))$-Gorenstein projective $R$-modules in the sense of \cite[Definition 3.2]{BMS}. Note that under this notion $\GP_{\A_C} \subseteq \GP$.
Finally we recall the following notion.  
 
 \begin{defi}[GP-admissible pair] \cite[Definition 3.1]{BMS}. A pair $(\X, \Y) \subseteq \Modu (R) \times \Modu (R)$ is  \textbf{GP-admissible} if satisfies the following conditions:
\begin{enumerate}
\item $\Ext^{\geq 1}_{R}(\mathcal{X,Y}) = 0$.

\item For every  $A \in \Modu (R)$ there is an epimorphism $X \to A$ with $X \in \mathcal{X}$.

\item $\mathcal{X}$ and $\mathcal{Y}$ are closed under finite coproducts.

\item $\mathcal{X}$ is closed under extensions.

\item $\mathcal{X} \cap \mathcal{Y}$ is a relative cogenerator in $\mathcal{X}$, that is, for every  $X \in \mathcal{X}$ there is an exact sequence $0\to X \to  W \to X' \to  0$ with $X' \in \mathcal{X}$ and $W \in \mathcal{X} \cap \mathcal{Y}$. 
\end{enumerate}
\end{defi}
We are ready to state a couple of results, from where we will obtain in each case a complete cotorsion pair, where from we will obtain a family of cotorsion pairs  based on an initial cotorsion pair.

\begin{teo} \label{Teo-AC}
Let $_SC _R$ be a semidualizing $(S,R)$-bimodule. The class $\GP _{\A _C} $ is special precovering under the following conditions.
\begin{itemize}
\item[(i)] When $\A  _C (R) \subseteq \GP _{\A _C} ^{\gorro}$. 

\item[(ii)] If $\pd (\mathcal{I}_C (R))\leq m$ and $\id (\A _C (R))  < \infty$.
\end{itemize}
Furthermore, in each case we obtain an hereditary and complete cotorsion pair $(\GP _{\A _C}, \GP _{\A _C}  ^{\ortogonal })$
\end{teo}
\begin{proof}
(i) The containment  $ \A _C (R) \subseteq \GP _{\A _C} ^{\gorro} $,  implies that  all $N \in \A _C (R)$ possesses a special $\GP _{\A _C}$-precover by \cite[Theorem 4.1]{BMS}, since $(\Proj (R), \A_C (R))$ forms a GP-admissible pair \cite[Definition 3.1]{BMS}. We know that the pair $(\A _C (R), \A _C (R) ^{\ortogonal})$ is a perfect cotorsion pair, then by \cite[Theorem 8]{Estrada23} we get that the class $\GP _{\A _C} $ is  special precovering, since  from \cite[Proposition 3.9]{EnJen-n-per} also $\GP \subseteq \A _C (R)$, and by definition $\GP _{\A _C} \subseteq  \GP $

(ii)  We will do the proof in two parts, to the first only assume that  $\pd (\mathcal{I}_C (R))\leq m$. Take $M \in \A_{C} (R)$. By \cite[Proposition 3.12]{EnJen-n-per}  there exists an exact sequence $$0 \to M \to U^{0} \to U^1 \to \cdots $$ where each $U^{i } \in \mathcal{I}_C (R)$.

 From the inequality  $\pd (\mathcal{I}_C (R))\leq m$ and \cite[Ch. XVII, \S 1 Proposition 1.3]{Cartan} we can construct the following commutative diagram
$$\xymatrix{  
 0\ar[r] & Q_{_{}} \ar@{^{(}->}[d] \ar@{^{(}->}[r] &  P^0 _{m_{}}  \ar@{^{(}->}[d] \ar[r]& P^1 _{m_{}}  \ar@{^{(}->}[d] \ar[r] & \cdots \\
0 \ar[r]& Q_{m-1} \ar[r] \ar[d] & P^0 _{m-1} \ar[d]  \ar[r]&  P^1 _{m-1} \ar[d]  \ar[r] & \cdots  \\
    &  \vdots  \ar[d] & \vdots  \ar[d] & \vdots \ar[d] & \\
 0\ar[r] & Q_0 \ar@{>>}[d] \ar[r] & P^0 _0  \ar@{>>}[d] \ar[r]& P^1 _0  \ar@{>>}[d] \ar[r] & \cdots \\
0 \ar[r]& M \ar@{^{(}->}[r]  & U^0   \ar[r] &  U^1 \ar[r]  & \cdots  }$$
such that  $P_i ^j \in \Proj (R)$ for all $i \in \{0,1, \dots , m \}$ and all $j\geq 0$. With  $Q _i := \Ker (P^0 _i \to P^1 _i)\in \Proj (R)$ for all $i \in \{0 , \dots , m-1\}$ and 
\begin{center}
$\Lambda  : 0 \to Q \to P_m ^0 \to P_m ^1 \to \cdots$\\
 $\Theta : 0 \to Q \to Q_{m-1} \to \cdots \to Q_0 \to M \to 0$
 \end{center}
 are acyclic complexes. Note that the complex $\Lambda$ can be completed on the left with a resolution by projective $R$-modules,  thus we obtain an acyclic complex by projectives $\widetilde{\Lambda}$. Note that $\widetilde{\Lambda}$ and $\Theta$  were constructed only with the condition  $\pd (\mathcal{I}_C (R))\leq m$. 
  
   To the second part we will assume  $\id (\A _C (R) ) < \infty$.  Then, by \cite[Lemma 3.6]{Becerril21} the acyclic complex $ \widetilde{\Lambda} $ is $\Hom _R (- , \A _C (R))$-exact, from where $Q \in \GP _{\A _C}$. But then, from the acyclic complex $\Theta$ we obtain that $\resdim _{\GP _{\A _C} }(M) \leq m$, that is $M \in { [\GP _{\A _C} ]}^{\gorro} _m$. We prove under this conditions that $\A  _C (R) \subseteq {[\GP _{\A _C}]}^{\gorro} _{m}$, and so the result follows from (i). 
   
   The furthermore part follows from  \cite[Proposition 29]{Gill21}. 
\end{proof} 
From the ideas we have developed so far, we obtain the following result.
\begin{teo}
Let $_S C _R$ be a semidualizing $(S, R)$-bimodule. Then the class $\GP$ is special precovering under the following situations. 
\begin{itemize}
\item[(i)] If $\A _C (R) \subseteq \Proj (R) ^{\gorro}$.
\item[(ii)]  If $\pd (\mathcal{I}_C (R))\leq m$ and $\Proj (R) \subseteq \mathcal{SFI} ^{\cogorro}$.
\end{itemize}
Furthermore, in each case we obtain an hereditary and complete cotorsion pair $(\GP , \GP  ^{\ortogonal })$
\end{teo}

\begin{proof}
(i) Indeed, the  condition $\A _{C}  (R) \subseteq \Proj (R) ^{\gorro}$ implies that $ \A _{C}  (R) \subseteq \GP   ^{\gorro}$, since every projective is Gorenstein projective.  Thus, from \cite[Theorem 8]{Estrada23} and since $(\A _C (R), \A _C (R) ^{\ortogonal})$ is a complete cotorsion pair, we get $\GP  $ is special precovering. 

(ii)  Take $M \in \A_{C} (R)$,  since $\pd (\mathcal{I}_C (R))\leq m$, then by the first part in Theorem  \ref{Teo-AC} (ii), we get an acyclic complex of projectives $ \widetilde{\Lambda} $, and another  $\Theta$. By Lemma \ref{Hom-exacto} the complex $ \widetilde{\Lambda} $ is $\Hom _R (-, \Proj (R))$-exact, since $\Proj (R) \subseteq \mathcal{SFI} ^{\cogorro}$. Then, in the acyclic  complex $\Theta$ we obtain $Q \in \GP$, that is $M \in \GP ^{\gorro}$. Thus, every $M \in \A_C (R)$ possesses a special $\GP$-precover, then  by \cite[Theorem 8]{Estrada23} we get that  $\GP$ is special precovering. 

The last part  follows from  \cite[Proposition 29]{Gill21}.
\end{proof}
The previous result allows us to obtain a family of cotorsion pairs  and  a family Hovey triples, which can be seen as a Gorenstein projective  counterpart of Rachid's work \cite{Rachid} from which the following statements are inspired,  but in a more general environment.

\begin{lem} \label{Lema01}
Let $R$ be a ring and $\Le \subseteq \Modu (R)$ such that $(\Proj (R), \Le)$ is a  GP-admissible pair.  For $M \in \Modu (R)$ and $m \in \mathbb{N}$ consider the following statements. 
\begin{itemize}
\item[(i)] $M \in [\GP _{\Le} ] ^{\gorro} _{m} $,
\item[(ii)] $\Ext ^1 _R ( M,E) =0$ for all $E \in [\Proj (R)^{\gorro} _m] ^{\ortogonal_1 } \cap \GP _{\Le} ^{\ortogonal _1}$.
\end{itemize}
Then  $\mathrm{(i)}\Rightarrow \mathrm{(ii)}$. If the pair  $(\GP _{\Le} , \GP _{\Le}  ^{\ortogonal _1})$ is a   complete cotorsion pair, then $\mathrm{(ii)}\Rightarrow \mathrm{(i)}$, and thus both conditions are equivalent. 
\end{lem} 

\begin{proof}
We will use the following equality $\Proj (R) ^{\gorro} _m = \GP _{\Le} ^{\ortogonal _1 } \cap [ \GP_{\Le}] ^{\gorro} _m$. From \cite[Corollaries  5.2 (b) and 4.3 (c) ]{BMS} we have the containment  $\Proj (R) ^{\gorro} _m \subseteq  \GP _{\Le} ^{\ortogonal _1 } \cap[ \GP_{\Le}] ^{\gorro} _m$. Now if $M \in  \GP _{\Le} ^{\ortogonal _1 } \cap [\GP _{\Le} ] ^{\gorro} _{m}$ then from \cite[Theorem 4.1 (b)]{BMS} there is an exact sequence $\gamma : 0 \to M \to H  \to G \to 0$ with $H \in \Proj (R) ^{\gorro} _{m}$ and $G \in \GP _{\Le}$, this implies that the previous sequence  $\gamma$ splits, since $M \in \GP _{\Le} ^{\ortogonal _1}$, and so $M $ is direct summand of $H \in \Proj(R) ^{\gorro} _m$. 

$\mathrm{(i)}\Rightarrow \mathrm{(ii)}$. By induction over $m$.  If $m =0$ is clear. Now assume that $m > 0 $. From \cite[Theorem 4.1 (a)]{BMS}  there is an exact sequence $ \theta :0 \to K \to G \to M \to 0$ with $K \in \Proj (R) ^{\gorro} _{m-1} $ and $G \in \GP _{\Le}$. By definition there is an exact sequence $0 \to G \to P \to G' \to 0$ with $P \in \Proj (R)$ and $G' \in \GP _{\Le}$. We can construct the following p.o digram
$$\xymatrix{ 
         K  \ar@{=}[d] \ar@{^{(}->}[r]  & G _{} \ar@{^{(}->}[d]  \ar@{}[dr] |{\textbf{po}} \ar@{>>}[r] & M _{}   \ar@{^{(}->}[d]    \\
  K \ar@{^{(}->}[r]  & P   \ar@{>>}[r] & Q }$$
Where $Q \in \Proj (R)^{\gorro} _m$. For $\overline{E}\in [\Proj (R)^{\gorro} _m] ^{\ortogonal _1 }$, applying $\Hom _R (-,\overline{E})$ we obtain the following commutative diagram 
$$\xymatrix{ 
         \Hom _R (M,\overline{E})   \ar@{^{(}->}[r]  & \Hom _R (G,\overline{E}) _{}     \ar[r] & \Hom _R (K,\overline{E})  _{} \ar@{=}[d]      &   \\
  \Hom _R (Q,\overline{E}) \ar@{^{(}->}[r] \ar[u]  & \Hom _R (P,\overline{E})   \ar[r] \ar[u] & \Hom _R (K,\overline{E}) \ar[r] &\Ext^1  _R (Q,\overline{E}) =0, }  $$
  Thus $\Hom _R (G,\overline{E})  \to  \Hom _R (K,\overline{E}) $ is an epimorphism for all $\overline{E}\in [\Proj (R)^{\gorro} _m] ^{\ortogonal _1}$.  Now assume that $E \in [\Proj (R)^{\gorro} _m] ^{\ortogonal _1} \cap \GP _{\Le} ^{\ortogonal _1}$. From the above and $\theta$ we have the exact sequence 
  $$\Hom _R (G,E)  \twoheadrightarrow  \Hom _R (K,E) \to  \Ext ^1  _R (M,E)  \to \Ext ^1 _R (G,E) =0 $$
  Therefore, we conclude  $\Ext _R^1 (M, E) =0$ for all $E \in [\Proj (R)^{\gorro} _m] ^{\ortogonal _1} \cap \GP _{\Le} ^{\ortogonal _1}$.
  
  $\mathrm{(ii)}\Rightarrow \mathrm{(i)}$.  Let us suppose that $(\GP _{\Le}, \GP _{\Le} ^{\ortogonal _1})$ is a complete cotorsion pair and that for $M \in  \Modu (R)$ we have  $\Ext _R^1 (M, E) =0$ for all $E \in [\Proj (R)^{\gorro} _m] ^{\ortogonal _1} \cap \GP _{\Le} ^{\ortogonal _1}$.  We will use the equivalence of  \cite[Corollary 4.3 (c)]{BMS}. For $M \in  \Modu (R)$ there is an exact sequence $0 \to M \to H \to Q \to 0$, with $H \in  \GP _{\Le} ^{\ortogonal _1}$ and $Q \in \GP _{\Le}$. Applying $\Hom _R (-, E)$ we have the exact sequence 
  $$ \Ext _R^1 (Q, E)  \to \Ext _R^1 (H, E) \to  \Ext _R^1 (M, E) =0$$
  where $\Ext _R^1 (Q, E) =0$ since $Q \in \GP _{\Le}$ and $E \in \GP _{\Le} ^{\ortogonal _1}$. Therefore $\Ext _R^1 (H, E) =0$ for all $E \in [\Proj (R)^{\gorro} _m] ^{\ortogonal _1} \cap \GP _{\Le} ^{\ortogonal _1}$, we will use this fact at the end. 
  
  From \cite[Theorem 7.4.6]{EnJen00}, for $H$ there is an exact sequence $0 \to K' \to H' \to H \to 0$, with $H' \in \Proj (R) ^{\gorro}_m$ and $K' \in [\Proj (R) ^{\gorro}_m] ^{\ortogonal _1}$. Note that $H, H' \in \GP _{\Le} ^{\ortogonal_1}$ (since $\Proj (R)  \subseteq \GP _{\Le} ^{\ortogonal _1}$ implies $\Proj (R) ^{\gorro}_m  \subseteq \GP _{\Le} ^{\ortogonal _1}$ by the dual of  \cite[Lemma 2.6]{BMS}). This implies that $K' \in \GP _{\Le} ^{\ortogonal _1}$, i.e. $K' \in [\Proj (R) ^{\gorro}_m] ^{\ortogonal _1} \cap \GP _{\Le} ^{\ortogonal _1} $, so that  $\Ext ^1 _{R} (H, K') =0$. That is $0 \to K' \to H' \to H \to 0$  splits, therefore $H \in \Proj (R) ^{\gorro}_m$. Thus, the exact sequence $0 \to M \to H \to Q \to 0$ fulfils with the conditions of \cite[Corollary 4.3 (c)]{BMS}. 
\end{proof}

With the above, we are ready to obtain the expected family of cotorsion pairs. 
\begin{teo} \label{Family}
Let $R$ be a ring and $\Le \subseteq \Modu (R)$ such that $(\Proj (R), \Le)$ is a  GP-admissible pair. If   $(\GP _{\Le} , \GP _{\Le}  ^{\ortogonal _1})$ is a complete cotorsion pair, then for each $m > 0$ the pair 
\[  \displaystyle{\left( [\GP _{\Le}] ^{\gorro} _m ,\;  [\Proj (R)^{\gorro} _m] ^{\ortogonal _1} \cap \GP _{\Le} ^{\ortogonal _1}  \right) ,}
\]
is a complete and hereditary cotorsion pair.
\end{teo}

\begin{proof}
From Lemma \ref{Lema01} we know that  following;
\begin{align*}
[\GP _{\Le}] ^{\gorro} _m & = {^{\ortogonal _1} \displaystyle{\left( [\Proj (R)^{\gorro} _m] ^{\ortogonal _1} \cap \GP _{\Le} ^{\ortogonal _1} \right) }} & \text{and} & &  [\Proj (R)^{\gorro} _m] ^{\ortogonal _1} \cap \GP _{\Le} ^{\ortogonal _1} & \subseteq { (  [\GP _{\Le}] ^{\gorro} _m )^{\ortogonal _1}}.
\end{align*}
For the other hand we have  that $\Proj (R)^{\gorro} _m \cup \GP _{\Le} \subseteq [ \GP _{\Le}] ^{\gorro} _m \cup \GP _{\Le} = [ \GP _{\Le}] ^{\gorro} _m$. Then, taking orthogonals
\[ \displaystyle{ \left(  [ \GP _{\Le}] ^{\gorro} _m     \right) ^{\ortogonal _1}  \subseteq \left(   \Proj (R)^{\gorro} _m \cup \GP _{\Le}  \right) ^{\ortogonal _1 } =  (\Proj (R)^{\gorro} _m )^{\ortogonal _1} \cap \GP _{\Le}  ^{\ortogonal _1} }.
\]
This implies that $  \left(  [ \GP _{\Le}] ^{\gorro} _m     \right) ^{\ortogonal _1} =  (\Proj (R)^{\gorro} _m )^{\ortogonal _1} \cap \GP _{\Le}  ^{\ortogonal _1} $. Since the class $ [\GP _{\Le}] ^{\gorro} _m$ is closed by kernels of epimorphisms \cite[Corollary 4.10]{BMS}, we conclude that such cotorsion pair is hereditary. 
It remains to prove that such a pair is complete. We know that $(\GP _{\Le} , \GP _{\Le}  ^{\ortogonal _1})$ is complete as well as $(\Proj (R) ^{\gorro}_m , [\Proj (R) ^{\gorro} _m] ^{\ortogonal _1})$ \cite[Theorem 7.4.6]{EnJen00}, thus for $M \in \Modu (R)$ there is an exact sequence $0 \to K \to H \to M \to 0 \to 0$ with $H \in \GP _{\Le}$  and $K\in \GP _{\Le} ^{\ortogonal _1} $. And to this $K$ there is  $0 \to K \to S \to E \to 0$ with $S \in [\Proj (R) ^{\gorro} _m ] ^{\ortogonal _1}$ and $E \in \Proj (R) ^{\gorro} _m$. Consider the following p.o. diagram 

$$\xymatrix{ 
         K _{} ^{} \ar@{^{(}->}[d]   \ar@{^{(}->}[r]  \ar@{}[dr] |{\textbf{po}}  & H _{} \ar@{^{(}->}[d]   \ar@{>>}[r] & M _{}      \ar@{=}[d]   \\
  S ^{} \ar@{^{(}->}[r]  \ar@{>>}[d]  & G   \ar@{>>}[r]  \ar@{>>}[d] & M \\
    E  \ar@{=}[r] & E    &  }$$
   Where $S \in \GP _{\Le} ^{\ortogonal _1} $, since $K \in \GP _{\Le} ^{\ortogonal _1} $ and $E \in  \Proj (R) ^{\gorro} _m \subseteq \GP _{\Le} ^{\ortogonal _1}$. That is $S \in [\Proj (R) ^{\gorro} _m ] ^{\ortogonal _1} \cap \GP _{\Le} ^{\ortogonal _1} $. Now, since $H \in \GP _{\Le} \subseteq [\GP _{\Le}] ^{\gorro} _m$ and $E \in \Proj (R) ^{\gorro} _m \subseteq  [\GP _{\Le}] ^{\gorro} _m$ it follows that $G \in  [\GP _{\Le}] ^{\gorro} _m $. Thus the exact sequence that works is $0 \to S \to G \to M \to 0$.
\end{proof}
 As an expected consequence we also obtain a Hovey triple in $\Modu (R)$.
\begin{cor} \label{CorFam}
Let $R$ be a ring and $\Le \subseteq \Modu (R)$ such that $(\Proj (R), \Le)$ is a  GP-admissible pair. If   $(\GP _{\Le} , \GP _{\Le}  ^{\ortogonal _1})$ is a complete cotorsion pair, then for each $m > 0$ there is a hereditary Hovey triple in $\Modu (R)$ given by 
\[  \displaystyle{\left( [\GP _{\Le}] ^{\gorro} _m ,\;  \GP _{\Le} ^{\ortogonal _1}, [\Proj (R)^{\gorro} _m] ^{\ortogonal _1}   \right) }
\]
Whose homotopy category is the stable category $[\GP _{\Le}] ^{\gorro} _m \cap  [\Proj (R)^{\gorro} _m] ^{\ortogonal _1} / ( [\Proj (R)^{\gorro} _m] \cap [\Proj (R)^{\gorro} _m] ^{\ortogonal _1})$. 
\end{cor}

\begin{proof}
We will apply \cite[Theorem 1.2]{Gill15}. Note that the hereditary and complete cotorsion pairs $([\GP _{\Le}] ^{\gorro} _m ,\;  [\Proj (R)^{\gorro} _m] ^{\ortogonal _1} \cap \GP _{\Le} ^{\ortogonal _1} )$ and $(\Proj (R)^{\gorro} _m, [\Proj (R)^{\gorro} _m] ^{\ortogonal _1})$ are compatible, since from the proof of Lemma \ref{Lema01} we know that $\Proj (R) ^{\gorro} _m = \GP _{\Le} ^{\ortogonal _1 } \cap [ \GP_{\Le}] ^{\gorro} _m$. While the last assertion follows from \cite[Theorem 6.21]{Stovi14}.
\end{proof}

\begin{rk}
Of course Theorem \ref{Family} is applicable to $\Le = \Proj (R)$, and so we return to $\GP$. Thus, we can see that $(\GP , \GP ^{\ortogonal})$ is an hereditary and complete cotorsion pair, if and only if there exists  a non-negative integer $m$ such that $([\GP]^{\gorro} _m , {[\GP]^{\gorro} _m} ^{\ortogonal}) )$ is an hereditary and complete cotorsion pair. Indeed, if $([\GP]^{\gorro} _m , {[\GP]^{\gorro} _m} ^{\ortogonal}) )$ is hereditary and complete. Then, from \cite[Theorem 8]{Estrada23}, also is $(\GP , \GP ^{\ortogonal})$, since trivially $\GP \subseteq \GP ^{\gorro} _m$ and all $M \in \GP ^{\gorro} _m$ possesses a special $\GP$-precover. The other implication is Theorem \ref{Family}. 
\end{rk}

\bigskip

\textbf{Acknowledgements} The author want to thank professor Raymundo Bautista (Centro de Ciencias Matemáticas - Universidad Nacional Autónoma de México) for several helpful discussions on the results of this article.\\

\textbf{Funding} The author was fully supported by a CONAHCyT (actually renamed SECIHTI)  posdoctoral fellowship CVU 443002, managed by the Universidad Michoacana de San Nicolas de Hidalgo at the Centro de Ciencias Matemáticas, UNAM.






\begin{thebibliography}{20}

\bibitem{AuRe95} M. Auslander, I. Reiten, S.O. Smalo, S.O. \textit{Representation Theory of Artin Algebras}. Cambridge Studies in
Advanced Mathematics 36, Cambridge University Press, Cambridge 1995.








\bibitem{Araya} T. Araya, R.  Takahashi, Y Yoshino,  \textit{Homological invariants associated to semi-dualizing bimodules}, Journal of Mathematics of Kyoto University, 2005, vol. 45, no 2, p. 287-306.


\bibitem{Becerril21} V. Becerril, \textit{Relative global Gorenstein dimensions}, J. of Algebra and its Applications (2022) DOI:10.1142/S0219498822502085



\bibitem{BMS} V. Becerril, O.  Mendoza, V. Santiago. \textit{Relative Gorenstein objects in abelian categories}, Comm.  Algebra (2020) DOI: 10.1080/00927872.2020.1800023





\bibitem{BennisGlobal} D. Bennis, N. Mahdou, \textit{Global gorenstein dimensions}, Proceedings of the American Mathematical Society, 2010, vol. 138, no 2, p. 461-465.

\bibitem{Bel} A. Beligiannis, \textit{Cohen-Macaulay modules, (co)torsion pairs and virtually Gorenstein algebras}. J. Algebra
288, 137-211 (2005).

\bibitem{Bel-Rei} A. Beligiannis, I. Reiten, \textit{Homological and homotopical aspects of torsion theories}, Mem. Amer. Math. Soc.
188, 883(2007).



\bibitem{Cartan} H. Cartan, S. Eilenberg, \textit{Homological Algebra},  Princeton Univ. Press, Princeton,  N.J., 1956.

\bibitem{Chris} L. W. Christensen, S. Estrada, P. Thompson, \textit{Five theorems on Gorenstein global dimensions} In: Leroy, A., Jain, S.K. (eds), Algebra and Encoding Theory, Contemp. Math. Vol 785. pp. 67-78. (Amer. Math. Soc., Providence RI,2023)

\bibitem{Holm06} L. W. Christensen, A. Frankild, H. Holm, \textit{On Gorenstein projective, injective and flat dimensions-a functorial description with applications} Journal of Algebra 2006, Vol 302, 231-279. 

\bibitem{Cortes} M. Cortés-Izurdiaga, J. Šaroch, \textit{The cotorsion pair generated by the Gorenstein projective modules and $\lambda$-pure-injective modules}. Israel Journal of Mathematics, 2025, p. 1-30.














\bibitem{EnJen93} E. E. Enochs, O. M. G. Jenda, \textit{Copure injective resolutions, flat resolvents and dimensions}, Commentationes Mathematicae Universitatis Carolinae, 1993, vol. 34, no 2, p. 203-211.


\bibitem{EnJen00} E. Enochs and O. M. G. Jenda, \textit{Relative Homological Algebra}, De Gruyter Expositions in Mathematics no. 30, Walter De Gruyter, New York, 2000.


\bibitem{EnJen-n-per} E. E. Enochs, O. M. G. Jenda, J. A.  L\'opez-Ramos, \textit{Dualizing modules and n-perfect rings}. Proceedings of the Edinburgh Mathematical Society, 2005, vol. 48, no 1, p. 75-90.

\bibitem{Emma24} I. Emmanouil, I. Kaperonis,  \textit{On K-absolutely pure complexes} Journal of Algebra, 2024, vol. 640, p. 274-299.

\bibitem{Emma25} I. Emmanouil, O. Talelli, \textit{Triviality criteria for unbounded complexes}, Journal of Algebra, 2025, vol. 663, p. 786-814.

\bibitem{Estrada23}  S. Estrada, A. Iacob, \textit{Gorenstein projective precovers and finitely presented modules},  Rocky Mountain Journal of Mathematics, 2024, vol. 54, no 3, p. 715-721.


\bibitem{EA17} S. Estrada, A. Iacob,  K. Yeomans, \textit{Gorenstein projective precovers}, Mediterranean Journal of Mathematics, 2017, vol. 14, no 1, p. 1-10.




\bibitem{Foxby} H.-B. Foxby  \textit{Gorenstein modules and related modules}. Math. Scand., 31:267-284, 1973.

\bibitem{Gedrich} T.V. Gedrich,  K.W Gruenberg, \textit{Complete cohomological functors on groups} Topology Appl. 25,
203-223 (1987).

\bibitem{Gill17} J. Gillespie, \textit{On Ding injective, Ding projective and Ding flat modules and complexes}, Rocky Mountain Journal of Mathematics, 2017, vol. 47, no 8, p. 2641-2673.

\bibitem{Gill15} J. Gillespie, \textit{How to construct a Hovey triple from two cotorsion pairs}, Fund. Math. 230(3)(2015), 281-289.

\bibitem{Gill16} J. Gillespie, \textit{Gorenstein complexes and recollements from cotorsion pairs}. Advances in Mathematics, 2016, vol. 291, p. 859-911.


\bibitem{Gill21} J. Gillespie, A. Iacob, \textit{Duality pairs, generalized Gorenstein modules, and Ding injective envelopes}, Comptes Rendus. Math\'ematique 360.G4 (2022): 381-398, DOI:https:10.5802/crmath.306





\bibitem{HolmWhite} H. Holm, D. White, \textit{Foxby equivalence over associative rings}, Journal of Mathematics of Kyoto University, 2007, vol. 47, no 4, p. 781-808.
\bibitem{Holm04-1} H. Holm,  \textit{Rings with finite Gorenstein injective dimension}, Proceedings of the American Mathematical Society, 2004, vol. 132, no 5, p. 1279-1283.




\bibitem{Huang} Z. Huang, \textit{Duality pairs induced by Auslander and Bass classes}, Georgian Mathematical Journal, 2021, vol. 28, no 6, p. 867-882.







\bibitem{Mar} R. Marczinzik, \textit{Gendo-symmetric algebras, dominant dimensions and Gorenstein homological algebra}, arXiv preprint arXiv:1608.04212, 2016.

 \bibitem{Rachid} R. El Maaouy, \textit{Model structures, $n$-Gorenstein flat modules and PGF dimensions}, Proceedings of the Edinburgh Mathematical Society, 2024, vol. 67, no 4, p. 1241-1264.


\bibitem{Guan}  W. Li, J. Guan, B. Ouyang, \textit{Strongly $FP$-injective modules}, Commun. Algebra 45 (2017) 3816-3824.

\bibitem{Ringel} C. M. Ringel,  P. Zhang, \textit{Gorenstein-projective and semi-Gorenstein-projective modules}, Algebra \& Number Theory, 2020, vol. 14, no 1, p. 1-36.

\bibitem{Stovi14} J.  \v{S}t'ov\'i\v{c}ek, \textit{Exact model categories, approximation theory, and cohomology of quasi-coherent sheaves},
Advances in Representation Theory of Algebras, EMS Series of Congress Reports, European Math. Soc. Publishing House, 2014, pp. 297-367. DOI:10.4171/125-1/10.

\bibitem{Wang-Li} J. F. Wang,  H. H. Li, \textit{When do the Gorenstein Injective Modules and Strongly Cotorsion Modules Coincide?}. Bulletin of the Iranian Mathematical Society, 2023, vol. 49, no 6, p. 82. 

\bibitem{Xu} J. Z. Xu, \textit{Minimal injective and flat resolution over Gorenstein rings}, J. Algebra 175, 451-477 (1995).

\bibitem{Stov} J. \v{S}aroch,  J.  \v{S}t'ov\'i\v{c}ek, \textit{Singular compactness and definability for $\Sigma$-cotorsion and Gorenstein modules}, Selecta Mathematica, 2020, vol. 26, no 2, p. 1-40.

 
 


 
\bibitem{Yang14}  G. Yang G, L. Liang  \textit{All modules have Gorenstein flat precovers. Communications in Algebra} 2014; 42: 3078-3085.
 





\end{thebibliography}
\end{document}